\documentclass[a4paper,12pt,twoside]{article}
\usepackage{amsmath,amsthm,amscd,amsfonts,amssymb,amsxtra}
\usepackage[hmargin=1.2in,vmargin=1.2in]{geometry}
\usepackage[utf8]{inputenc}
\usepackage{graphicx}
\usepackage{microtype}
\usepackage[auth-lg,affil-it]{authblk}
\usepackage[pdftex,bookmarks,colorlinks=false]{hyperref}
\usepackage{booktabs}
\usepackage{enumitem}
\usepackage{caption,subcaption}
\usepackage[mathscr]{euscript}

\newcommand{\sP}{\mathscr{P}}
\newcommand{\mZ}{\mathbb{Z}_n}

\newtheorem{thm}{Theorem}[section]

\newtheorem{cor}[thm]{Corollary}
\newtheorem{defn}[thm]{Definition}

\newtheorem{prob}{Problem}
\newtheorem{prop}[thm]{Proposition}
\newtheorem{rem}[thm]{Remark}

\def\ni{\noindent}

\title{\sc A Study on the Modular Sumset Labeling of Graphs}

\author{{\bf Sudev Naduvath}}
\affil{\small Centre for Studies in Discrete Mathematics\\ Vidya Academy of Science \& Technology \\ Thrissur, Kerala, India.\\ E-mail: sudevnk@gmail.com}

\date{}

\begin{document}
\maketitle


\begin{abstract}
For a positive integer $n$, let $\mZ$ be the set of all non-negative integers modulo $n$ and $\sP(\mZ)$ be its power set. A modular sumset valuation or a modular sumset labeling of a given graph $G$ is an injective function $f:V(G) \to \sP(\mZ)$ such that the induced function $f^+:E(G) \to \sP(\mZ)$ defined by $f^+ (uv) = f(u)+ f(v)$. A modular sumset indexer of a graph $G$ is an injective modular sumset valued function $f:V(G) \to \sP(\mZ)$ such that the induced function $f^+:E(G) \to \sP(\mZ)$ is also injective. In this paper, some properties and characteristics of this type of modular sumset labeling of graphs are being studied.
\end{abstract}

\ni {\small \bf Key Words}: Modular sumset graphs; weak modular sumset graphs; strong modular sumset graphs, maximal modular sumset graphs exquisite modular sumset graphs; modular sumset number of a graph. 

\vspace{0.53mm}

\ni {\small \bf Mathematics Subject Classification}: 05C78.

\section{Introduction}

For all  terms and definitions, not defined specifically in this paper, we refer to \cite{BM1}, \cite{FH} and \cite{DBW}. For graph classes, we further refer to \cite{BLS}, \cite{JAG} and \cite{GCO} and for notions and results in number theory, we refer to \cite{TMA} and \cite{MBN}.  Unless mentioned otherwise, all graphs considered here are simple, finite and have no isolated vertices.

Let $A$ and $B$ be two sets integers. The {\em sumset} of $A$ and $B$ is denoted by $A+B$ and is defined as $A+B=\{a+b : a\in A, b\in B\}$. If either $A$ or $B$ is countably infinite, then $A+B$ is also countably infinite. We denote the cardinality of a set $A$ by $|A|$. Then, we have the following theorem on the cardinality of the sumset of two sets. 

\begin{thm}\label{T-CST}
\cite{MBN} For two non-empty sets $A$ and $B$, $|A|+|B|-1\le |A+B|\le |A|\,|B|$.
\end{thm}

Another theorem on sumsets of two sets of integers proved in \cite{MBN} is given below.

\begin{thm}\label{T-SSAPL}
Let $A$ and $B$ be two non-empty sets of integers. Then, $|A+B|=|A|+|B|-1$ if and only if $A$ and $B$ are arithmetic progressions with the same common difference.
\end{thm}

Using the concepts of the sumset of two sets, the notions of an integer additive set-labeling and an integer additive set-indexer of a given graph $G$ was introduced as follows.

\begin{defn}{\rm
Let $\mathbb{N}_0$ be the set of all non-negative integers and let $\sP(\mathbb{N}_0)$ be its power set. An {\em integer additive set-labeling} of a graph $G$ is an injective function $f:V(G)\to \sP(\mathbb{N}_0)$, where the induced function $f^+(uv):E(G)\to \sP(\mathbb{N}_0)$ is defined by $f^+(uv)=f(u)+f(v)$. A graph $G$ which admits an integer additive set-labeling is called an {\em integer additive set-labeled graph} or {\em integer additive set-valued graph}.}
\end{defn}

\begin{defn}{\rm
\cite{GA} An {\em integer additive set-indexer} (IASI) is  as an injective function $f:V(G)\to \sP(\mathbb{N}_0)$ such that the induced function $f^+:E(G) \to \sP(\mathbb{N}_0)$ defined by $f^+ (uv) = f(u)+ f(v)=\{a+b: a \in f(u), b \in f(v)\}$ is also injective. A graph $G$ which admits an IASI is called an {\em integer additive set-indexed graph}.}
\end{defn}

Certain studies about integer additive set-indexers of graphs have been initiated in \cite{GA}, \cite{GS1}, \cite{GS0} and \cite{GS2}. A series of studies about different types of integer additive set-labeled graphs followed thereafter. These papers about the characteristics and properties of integer additive set-valued graphs are the main motivations behind this paper.   

\section{Modular Sumset Labeling of Graphs}

Let $n$ be a positive integer. We denote the set of all non-negative integers modulo $n$ by $\mZ$ and its power set by $\sP(\mZ)$. The {\em modular sumset} of the two subsets $A$ and $B$ of $\mZ$, denoted by $A+B$, is the set defined by $A+B=\{x: a+b\equiv x\;(mod\; n), a\in A, b\in B\}$. Through out our discussion, $A+B$ is the sumset of $A$ and $B$.  It can also be noted that $A,B \subseteq \mZ \implies A+B \subseteq \mZ$. 

Then, using the concepts of modular sumsets of sets and analogous to the definition of integer additive set-labelings of graphs, let us now define the following notions. 

\begin{defn}{\rm
A function $f:V(G)\to \sP(\mZ)$, whose induced function $f^+(uv):E(G)\to \sP(\mZ)$ is defined by $f^+(uv)=f(u)+f(v)$, is said to be a {\em modular sumset labeling} if $f$ is injective. A graph $G$ which admits a modular sumset labeling is called an {\em modular sumset graph}.}
\end{defn}

\begin{defn}{\rm
A {\em modular sumset indexer} is  an injective function $f:V(G)\to \sP(\mZ)$ such that the induced function $f^+:E(G) \to \sP(\mZ)$ defined by $f^+ (uv) = f(u)+ f(v)=\{x: a+b\equiv x\;(mod\;n); a \in f(u), b \in f(v)\}$ is also injective.}
\end{defn}

In View of Theorem \ref{T-CST}, the bounds for the cardinality of an edge of a modular sumset graph $G$ is given in the following theorem.

\begin{thm}\label{T-CST2}
Let $f:V(G)\to \sP(\mZ)$ be a modular sumset labeling of a given graph $G$. Then, for any edge $uv\in E(G)$, we have $|f(u)|+|f(v)|-1\le |f^+(uv)|=|f(u)+f(v)|\le |f(u)|\,|f(v)|\le n$.
\end{thm}

\ni We use the following terms and definitions analogous to those for IASL-graphs. The cardinality of the set-label of an element of $G$ is said to be the {\em set-labeling number} of that element. If all the vertices of a graph $G$ have the same set-labeling number, say $l$, then we say that $V(G)$ is {\em $l$-uniformly set-labeled}. A modular sumset labeling of $G$ is said to be a {\em $k$-uniform modular sumset labeling} if the set-labeling number of all edges of $G$ have the same set-labeling number $k$.

The first and most important problem in this context is to verify whether a given graph admits a modular sumset labeling. The existence of a modular sumset labeling (and a modular sumset indexer) for any given graph $G$, is established in the following theorem.

\begin{thm}
Every finite graph admits a modular sumset labeling (or a modular sumset indexer). 
\end{thm}
\begin{proof}
Let $G$ be a graph with vertex set $V(G) = \{v_1, v_2, \ldots, v_p\}$.. Let $A_1,A_2\ldots,A_p$ be $p$ distinct non-empty subsets of $\mZ$, $n$ being a positive integer. Let $f : V (G) \to \sP(\mZ)$ defined by $f(v_i) = A_i$. Clearly $f$ is injective. Define $f^+ : E(G) \to \sP(\mZ)$ by $f^+(v_iv_j) = \{x: a_i +b_j\equiv x\;(mod\;n); a_i \in A_i; b_j \in A_j\}$. Clearly, $f^+(v_iv_j)$ is also a set of non-negative integers modulo $n$ and hence is a subset of $\mZ$. Therefore, for suitable choices of $n$, $f$ is a modular sumset labeling of $G$.

For taking $n$ sufficiently large, we can make $f^+$ also an injective function, which establishes the existence of a modular sumset indexer for $G$. 
\end{proof}

Figure \ref{G-SSLG} depicts the existence of a modular sumset labeling for the Petersen Graph, with respect to the ground set $\mathbb{Z}_4$ and Figure \ref{G-SSIG} illustrates the existence of a modular sumset indexer for the Petersen Graph, with respect to the ground set $\mathbb{Z}_5$.

\begin{figure}[h!]
\centering
\begin{subfigure}[b]{0.45\textwidth}
\includegraphics[width=\textwidth]{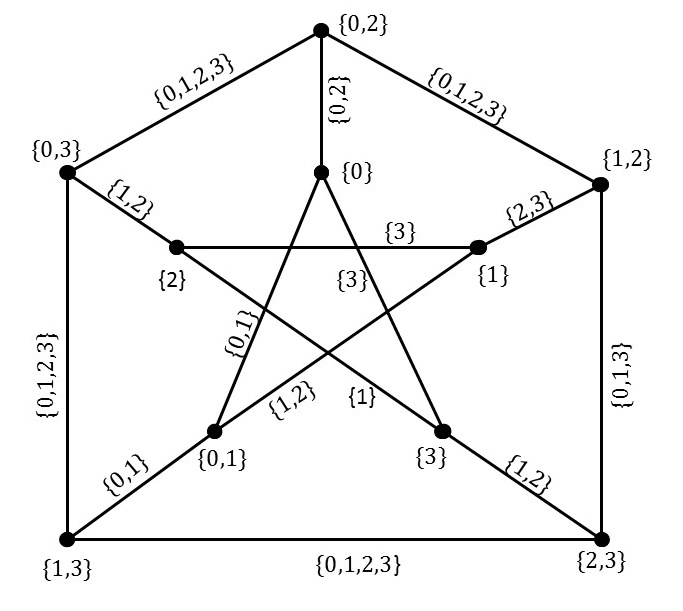}
\caption{\small Petersen Graph together with a sumset valuation defined on it}
\label{G-SSLG}
\end{subfigure}
\qquad
\begin{subfigure}[b]{0.45\textwidth}
\includegraphics[width=\textwidth]{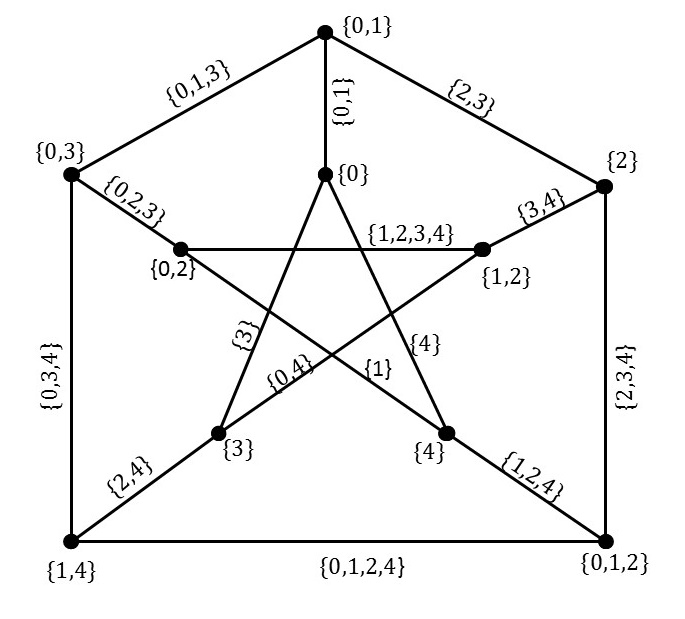}
\caption{\small Petersen Graph together with a modular sumset indexer defined on it}
\label{G-SSIG}
\end{subfigure}
        \caption{}\label{}
\end{figure}

As in the case of all other set-valuations of graphs, the problem of finding the minimum cardinality required for the ground set $\mZ$ with respect to which a given graph admits a modular sumset labeling is relevant and arouses much interest. In view of this, let us introduce the following notion.

\begin{defn}{\rm
The smallest value of $n$ such that $f:V(G)\to \sP(\mZ)$ is a modular sumset labeling of a given graph $G$ is called the \textit{modular sumset number} of $G$. The modular sumset number of a graph $G$ is denoted by $\sigma(G)$. Similarly, the minimum value of $n$ such that $f:V(G)\to \sP(\mZ)$ is a modular sumset indexer of $G$ is called the \textit{sumset index} of $G$ and is denoted by $\bar{\sigma}(G)$.}
\end{defn}

\ni The following theorem establishes the modular sumset number of a given graph on $m$ vertices.

\begin{thm}
The modular sumset number of a graph $G$ is $1+\lfloor \log_2 m\rfloor$.
\end{thm}
\begin{proof}
Let $f:V(G)\to \sP(\mZ)$ be a modular sumset labeling of a given graph $G$ on $m$ vertices. Then each vertex of $G$ has a non-empty subset of $\mZ$ as its set-label. Therefore, the ground set $\mZ$ must have at least $m$ non-empty subsets. That is, $2^n-1\ge m$. Hence, $n\ge 1+\log_2m$. Therefore, $\sigma(G)=1+\lfloor \log_2 m\rfloor$.
\end{proof}

\section{Certain Types Modular Sumset Graphs}

In this section, we study certain types of modular sumset labelings of given graphs according to the nature of the set-labels of the elements of those graphs. 

Analogous to the studies of integer additive set-labeled graphs, let us now proceed to study different types of modular sumset labelings of $G$ in terms of the cardinality of the set-labels of vertices and edges of $G$. 

\subsection{Weak Modular Sumset Graphs}

\ni As in the study of weak IASL graphs, our first aim is to check the existence of edges in modular sumset graph which has the same set-labeling number as that of one or both of its end vertices. The following result establishes the condition for a sumset to have the same cardinality of one or both of its summands. 

\begin{prop}
Let $A$ and $B$ be two non-empty subsets of $\mZ$. Then, $|A+B|=|A|$ (or $|A+B|=|B|$) if and only if either $|A|=|B|=\mZ$ or $|B|=1$ (or $|A|=1$). More over, $|A+B|=|A|=|B|$ if and only if $|A|=|B|=\mZ$ or $|A|=|B|=\mZ$.
\end{prop} 

In view of the above proposition, it can be noted that the set-labeling number of an edge of $G$ is equal to that of an end vertex of that edge if and only if the set-labeling number of the other end vertex is $1$. Hence, analogous to the corresponding notion of IASL-graphs (see \cite{GS1,GS0}), we have the following definition. 

\begin{defn}{\rm
A modular sumset labeling $f$ of a graph $G$ is said to be a \textit{weak modular sumset labeling} of $G$ if the set-labeling number of every edge of $G$ is equal to the set-labeling number of at least one of its end vertices. A graph $G$ which admits a weak modular sumset labeling is called a {\em weak modular sumset graph}.}
\end{defn} 

It is to be noted that for a weak modular sumset graph, no two adjacent vertices can have non-singleton set-labels.

\begin{rem}{\rm 
The elements of $G$ having the set-labeling number $1$ are called the \textit{monocardinal elements} of $G$. An edge can be monocardinal if and only if its end vertices are monocardinal. The set-labeling number of an edge of a given graph $G$ is equal to the set-labeling number of  both of its end vertices if and only if the edge and its end vertices are monocardinal}
\end{rem} 

Hence, analogous to the corresponding of weak IASL-graphs the following results  are valid for weak modular sumset graphs.

\begin{thm}\label{T-wSL-1}
A graph $G$ admits a weak modular sumset labeling if and only if $G$ is bipartite or contains monocardinal edges. 
\end{thm}

\begin{thm}\label{T-wSL-2}
A graph $G$ admits a weakly uniform modular sumset labeling if and only if $G$ is bipartite. 
\end{thm}

\ni Figure \ref{fig:G-WSL} illustrates a weak modular sumset graph with respect to the ground set $\mathbb{Z}_6$.

\begin{figure}[h!]
\centering
\includegraphics[width=0.45\linewidth]{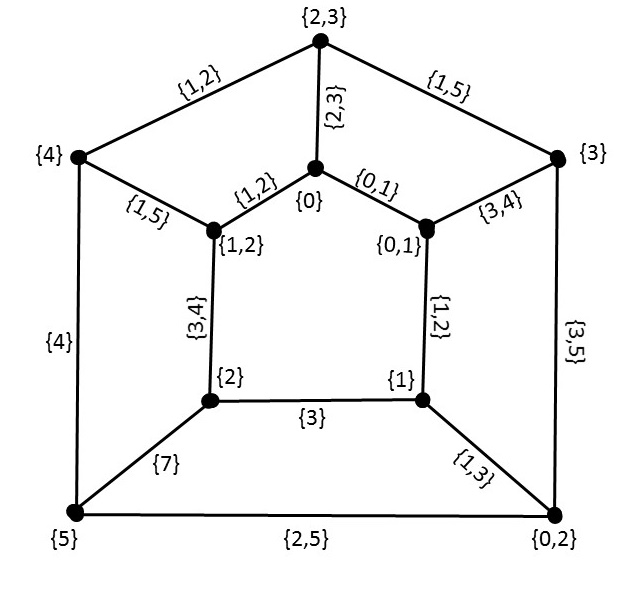}
\caption{An example to a weak modular sumset graph.}
\label{fig:G-WSL}
\end{figure}

It can be noticed that the study about weak modular sumset graphs is very much close to the study of weak IASL-graphs and hence offers no novelty in this context.

\subsection{Weak Modular Sumset Number of Graphs}

\ni As a special case of modular sumset number of graph we introduce the following notion.
\begin{defn}{\rm
The \textit{weak modular sumset number} of a graph $G$ is defined to be the minimum value of $n$ such that a modular sumset labeling $f:V(G)\to \sP(\mZ)$ is a weak modular sumset labeling of $G$.}
\end{defn}

In this section, we initiate a study about the weak modular sumset number of different graph classes. First we observe the following proposition.

In the following theorem, we determine the weak sumset number of an arbitrary graph $G$ in terms of its covering and independence numbers. 

\begin{thm}\label{T-WSL-G}
Let $G$ be a modular sumset graph and $\alpha$ and $\beta$ be its covering number and independence number respectively. Then, the weak modular sumset number of $G$ is $\max \{\alpha, r\}$, where $r$ is the smallest positive integer such that~~ $2^{r}-r-1\ge \beta$.
\end{thm}
\begin{proof}
First recall the result $\alpha(G)+\beta(G)=|V(G)|$. Since $G$ is a modular sumset graph, no two adjacent vertices can have non-singleton set-labels simultaneously. Therefore, the maximum number of vertices that have non-singleton set-labels is $\beta$. Let $V'$ be the set of these independent vertices in $G$. Therefore, the minimum number of vertices that are monocardinal is $|V(G)|-\beta=\alpha$. Since all these vertices in $V-V'$ must have distinct singleton set-labels, the ground set must have at least $\alpha$ elements. 

Also, the number non-empty, non-singleton subsets of the ground set must be greater than or equal to $\alpha$, since, otherwise all the vertices in $V'$ can not be labeled by non-singleton subsets of this ground set. Note that the number of non-empty, non-singleton subsets of a set $A$ is $2^{|A|}-|A|-1$.

Therefore, the weak modular sumset number $G$ is $\alpha$ if $2^{\alpha}-\alpha -1\ge \beta$. Otherwise, the ground set must have at least $r$ elements such that $2^{r}-r -1\ge \beta$. Therefore, in this case, the weak modular sumset number of $G$ is $r$, where $r$ is the smallest positive integer such that  $2^{r}-r -1\ge \beta$. Hence, $\sigma^{\ast}(G)=\max \{\alpha,  r\}$.
\end{proof}

We shall now discuss the weak modular sumset number of certain standard graph classes. First consider a path graph $P_m$ on $m$ vertices.

\begin{prop}
The weak modular sumset number of a path $P_m$ on $m$ vertices is given by 
\begin{equation*}
\sigma^{\ast}(P_m)=
\begin{cases}
m & \text{if}~~ m\le 2\\
\lfloor \frac{m}{2} \rfloor & \text{if}~~ m>2
\end{cases}
\end{equation*}
\end{prop}
\begin{proof}
Let $P_m$ denote be a path on $m$ vertices. Then, $\alpha=\lfloor \frac{m}{2} \rfloor$ and $\beta=\lceil \frac{m}{2} \rceil$. It is obvious that we need a singleton set to label the single vertex of $P_m$ if $m=1$ and a two element set to label the two vertices of $P_m$ when $m=2$. Hence, let $m\ge 3$. For any positive integer $m\ge 3$, we have $2^{\lfloor \frac{m}{2} \rfloor}>2$. Therefore, $2^{\lfloor \frac{m}{2} \rfloor}+\lfloor \frac{m}{2} \rfloor-1 > \lceil \frac{m}{2} \rceil$. That is, $2^{\alpha}-\alpha-1>\beta$. Hence, by Theorem \ref{T-WSL-G}, $\sigma^{\ast}(P_m)=\alpha=\lfloor \frac{m}{2} \rfloor$.
\end{proof}

\ni Next graph we consider is a cycle $C_m$ on $m$ vertices. The weak modular sumset number a cycle $C_m$ is given in the following result.

\begin{prop}
The weak modular sumset number of a cycle $C_m$, on $m$ vertices is 
\begin{equation*}
\sigma^{\ast}(C_m)=
\begin{cases}
m-1 & \text{if}~~ n=3, 4\\
\lceil \frac{m}{2} \rceil & \text{if}~~ n\ge 5
\end{cases}.
\end{equation*}
\end{prop}
\begin{proof}
For $C_3$, $\alpha=2$ and $\beta=1$. Therefore, $2^{\alpha}-\alpha-1=1=\beta$. Therefore, by Theorem \ref{T-WSL-G}, $\sigma^{\ast}(C_3)=2$. If $C_4$, $\alpha=2$ and $\beta=2$. Now, $2^{\alpha}-\alpha-1=1<\beta$. Also, for $r=3$, $2^r-r-1=2^3-3-1=4>\beta$. Therefore, by Theorem \ref{T-WSL-G}, $\sigma^{\ast}(C_4)=3$ 

For a cycle $C_m;~m\ge 5$, we have $\alpha=\lceil \frac{m}{2} \rceil$ and $\beta=\lfloor \frac{m}{2} \rfloor$. Since $m\ge 3$, $2^{\lceil \frac{m}{2} \rceil}\ge 4$. Therefore, $2^{\lceil \frac{m}{2} \rceil}+\lceil \frac{m}{2} \rceil-1 > \lfloor \frac{m}{2} \rfloor$. That is, $2^{\alpha}-\alpha-1>\beta$. Hence, by Theorem \ref{T-WSL-G}, $\sigma^{\ast}(C_m)=\alpha=\lceil \frac{m}{2} \rceil$.
\end{proof}

Another graph that can be generated from a cycle is a \textit{wheel graph}, denoted by $W_{n+1}$, which is the graph obtained by joining every vertex of a cycle $C_n$ to an external vertex. That is, $W_{n+1}=C_n+K_1$. The following result provides the weak modular sumset number of a wheel graph.

\begin{prop}
The weak modular sumset number of a wheel graph $W_{m+1}$ is $1+\lceil \frac{m}{2} \rceil$.
\end{prop}
\begin{proof}
For a wheel graph $W_{m+1}=C_m+K_1$, we have $\alpha=1+\lceil \frac{m}{2} \rceil$ and $\beta=\lfloor \frac{m}{2} \rfloor$. Since $m\ge 3$, as explained in the previous theorems, $2^{\alpha}-\alpha-1>\beta$. Hence, by Theorem \ref{T-WSL-G}, $\sigma^{\ast}(W_{m+1})=1+\lceil \frac{m}{2} \rceil$.
\end{proof}

Another graph that is close to cycles and wheel graphs is a \textit{helm graph}, denoted by $H_m$, which is the graph obtained by adjoining a pendant edge to each vertex of the outer cycle $C_m$ of a wheel graph $W_{m+1}$ . It has $2m+1$ vertices and $3m$ edges. The next result establishes the weak modular sumset number of a helm graph.

\begin{prop}
The weak modular sumset number of a helm graph $H_m$ is $m$.
\end{prop}
\begin{proof}
For $H_m$, we have $\alpha=m$ and $\beta=m+1$. It is to be noted that $2^{m-1}\ge (m+1)$ for all positive integer $m\ge 3$. Therefore, $2^m-m-1>m+1$. That is, $2^{\alpha}-\alpha-1>\beta$. Hence by Theorem \ref{T-WSL-G}, $\sigma^{\ast}(H_m)=\alpha=m$.
\end{proof}

Another graph which are related to paths are ladder graph, denoted by $L_m$ and is defined as the Cartesian product of a path on $m$ vertices and a path on two vertices. That is, $L_m= p_m\Box P_2$ (or $L_m=P_m\Box K_2$), where $m>2$. The weak modular sumset number of a ladder graph is determined in the following proposition.

\begin{prop}
The weak modular sumset number of a ladder graph $L_m$ is $m$.
\end{prop}
\begin{proof}
Let $L_m=P_m\Box P_2$, where $m>2$. Irrespective of whether $m$ is odd or even, $\alpha=m$ and $\beta=m$. For any positive integer $m>2$, we have $2^m-m-1>m$. That is, $2^\alpha-\alpha-1>\beta$. Then, by Theorem \ref{T-WSL-G}, $\sigma^{\ast}(L_n)=m$.
\end{proof}

\ni Next we proceed to determine the weak modular sumset number of a complete graph $K_n$. is determined in the following result. We already mentioned the weak modular sumset number of $K_1$, $K_2$ and $K_3$ when we discussed the weak modular sumset number of paths and cycles. Hence, need consider complete graphs having $4$ or more vertices. The weak modular sumset number of a complete graph $K_m$, where $m\ge 4$,  is determined in the following result.

\begin{prop}
For a positive integer $m\ge 4$, the weak modular sumset number of a complete graph $K_m$ is $m-1$.
\end{prop}
\begin{proof}
For a complete graph $K_m$, we have $\alpha=m-1$ and $\beta=1$. Therefore, for any positive integer $m$, we have $2^{m-1}-m-2>1$. That is, $2^{\alpha}-\alpha-1>\beta$. Therefore, the weak modular sumset number of $K_m$ is $m-1$.
\end{proof}

The minimum cardinality of the ground set when the given graph $G$ admits a weakly uniform modular sumset labeling arouses much interest in this occasion. Hence, we have the following result.

\begin{thm}\label{T-WSL-BP}
Let $G$ be a graph with covering number $\alpha$ and independence number $\beta$ and let $G$ admits a  weakly $k$-uniform modular sumset labeling, where $k<\alpha$ being a positive integer. Then, the minimum cardinality of the ground set $\mZ$ is $\max\{\alpha,r\}$, where $r$ is the smallest positive integer such that $\binom{r}{k}\ge \beta$.
\end{thm}
\begin{proof}
Let $G$ be a graph which admits a weakly $k$-uniform sum set labeling $f$ $A=\mZ$. Then, by Theorem \ref{T-wSL-2}, $G$ is bipartite. Let $X,Y$ be the bipartition of the vertex set $V(G)$. If $|X|\le |Y|$, then $\alpha=|X|$ and $\beta=|Y|$. Then, distinct elements of $X$ must have distinct singleton set-labels. Therefore, $n\ge \alpha$. 

On the other hand, since $f$ is $k$-uniform, all the elements in $Y$ must have distinct $k$-element set-labels. The number of $k$-element subsets of a set $A$ (obviously, with more than $k$ elements) is $\binom{|A|}{k}$. The ground set $A$ has $\alpha$ elements only if $\binom{\alpha}{k}\ge \beta$. Otherwise, the ground set $A$ must contain at least $r$ elements, where $r>\alpha$ is the smallest positive integer such that $\binom{r}{k} \ge \beta$. Therefore, $n=\max\{\alpha,r\}$.   
\end{proof}

In the above theorems we considered the value of $k$ which is less than or equal to $\alpha$. What is the case when $k\ge \alpha$? the following theorem provides a solution to this problem.

\begin{cor}
Let $G$ be a weakly $k$-uniform modular sumset graph, where $k\ge \alpha$, where $\alpha$ is the covering number of $G$. Then, the minimum cardinality of the ground set is the smallest positive integer $n$ such that $\binom{n}{k}\ge \beta$,  where $\beta$ is the independence number of $G$.   
\end{cor}
\begin{proof}
If $k\ge \alpha$, then $\binom{\alpha}{k}$ does not exist. Then, the proof is immediate from Theorem \ref{T-WSL-BP}.  
\end{proof}

Now, we proceed to verify the properties of certain graphs in which the set-labeling number of edges are the product of the set-labeling numbers of their end vertices. 

\subsection{Strong Modular Sumset Graphs}

\ni Analogous to strong IASL-graphs, we define the following notion.

\begin{defn}{\rm 
Let $f:V(G)\to \sP(\mZ)$ be a modular sumset labeling defined on a given graph $G$. Then, $f$ is said to be a {\em strong modular sumset labeling} if for the associated function $f^+:E(G)\to \sP(\mZ)$, $|f^+(uv)|=|f(u)|\,|f(v)|~ \forall~uv\in E(G)$. A graph which admits a strong modular sumset labeling is called a {\em strong modular sumset graph}.}
\end{defn}

Let $D_A$ be the \textit{difference set} of a given set $A$ defined by $D_A=\{|a_i-a_j|:a_i,a_j\in A\}$. Then, the necessary and sufficient condition for a graph to admit a strong modular sumset labeling is as given below.

\begin{thm}
A modular sumset labeling $f:V(G)\to \sP(\mZ)$ of a given graph $G$ is a strong modular sumset labeling of $G$ if and only if $D_{f(u)}\cap D_{f(v)}=\emptyset, ~\forall ~ uv\in E(G)$, where $|f(u)|\,|f(v)|\le n$. 
\end{thm}
\begin{proof}
Let $f:V(G)\to \sP(\mZ)$ be a modular sumset labeling on a given graph $G$. For any vertex $u\in V(G)$, define $D_f(u)=\{a_i-a_j: a_i,a_j\in f(u)\}$. 

Let $uv$ be an arbitrary edge in $E(G)$. Assume that $f$ is a strong modular sumset labeling of $G$. Then, by definition $|f^+(uv)|=|f(u)|\,|f(v)|$. Therefore, for any elements $a_i,a_j \in f(u)$ and $b_r,b_s\in f(v)$, we have $a_i+b_r\ne a_j+b_s$ in $f^+(uv)~\forall~uv\in E(G)$. That is, $|a_i-a_j|\ne |b_s-b_r|$ for any $a_i,a_j\in f(u)$ and $b_r,b_s\in f(v)$. That is, $D_{f(u)}\cap D_{f(v)}=\emptyset$. Therefore, the difference sets of the set-label of any two adjacent vertices are disjoint.

Conversely, assume that the difference  $D_{f(u)}\cap D_{f(v)}=\emptyset$ for any edge $uv$ in $G$. That is, $|a_i-a_j|\ne |b_s-b_r|$ for any $a_i,a_j\in f(u)$ and $b_r,b_s\in f(v)$. Then, $a_i-a_j\ne b_s-b_r$. That is, $a_i+b_r\ne a_j+b_s$. Therefore, all elements in $f(u)+f(v)$ are distinct. That is, $|f^+(uv)|=|f(u)|\,|f(v)|$ for any edge $uv\in E(G)$.  Hence, $f$ is a strong modular sumset labeling of $G$. 

Also, note that the maximum possible cardinality in the set-label of any element of $G$ is $n$, the product $|f(u)|\,|f(v)|$ can not exceed the number $n$. This completes the proof.
\end{proof}

Figure \ref{fig:G-SSL} illustrates a strong modular sumset labeling of a graph with respect to the ground set $\mathbb{Z}_6$.

\begin{figure}[h!]
\centering
\includegraphics[width=0.55\linewidth]{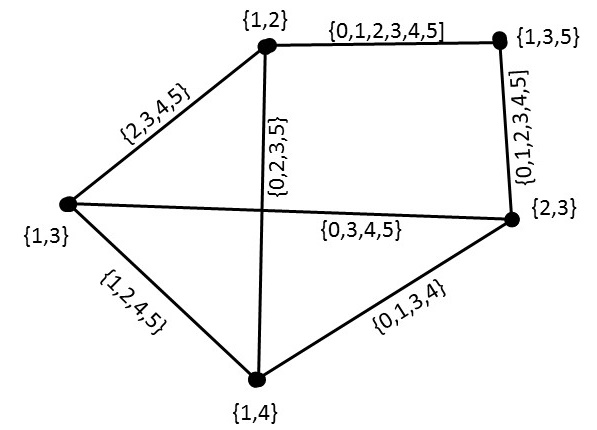}
\caption{An example to a strong modular sumset graph.}
\label{fig:G-SSL}
\end{figure}

Analogous to the weak modular sumset number of a graph $G$, we can define the \textit{strong modular sumset number} of $G$ as the minimum cardinality required for the ground set $\mZ$ so that $G$ admits a strong modular sumset labeling. The choice of ground set $\mZ$ is very important in this context because $n$ should be sufficiently large so that the vertices of the given graph can be labeled in such a way that the difference sets of these set-labels of all adjacent vertices must be pairwise disjoint.

Analogous to the corresponding results on strong IASL-graphs and strongly uniform IASL-graphs proved in \cite{GS2}, the following results are valid for strong modular sumset graphs also.

\begin{thm}\label{T-SSG2}
For a positive integer $k\le n$, a modular sumset labeling $f:V(G)\to \sP(\mZ)$ of a given connected graph $G$ is a a strongly $k$-uniform modular sumset labeling of $G$ if and only if either $k$ is a perfect square or $G$ is bipartite.
\end{thm}
\begin{proof}
If $k$ is a perfect square, say $k=l^2$, then we can label all the vertices of a graph by distinct $l$-element sets in such a way that the difference sets of the set-labels of every pair of adjacent vertices are disjoint. Hence, assume that $k$ is not a perfect square.   

Let $G$ be a bipartite graph with bipartition $(X,Y)$. Let $r,s$ be two divisors of $k$. Label all vertices of $X$ by distinct $r$-element sets all of whose difference sets are the same, say $D_X$. Similarly, label all vertices of $Y$ by distinct $s$-element sets all of whose difference sets the same, say $D_Y$, such that $D_X\cap D_Y=\emptyset$. Then, all the edges of $G$ has the set-labeling number $k=rs$. Therefore, $G$ is a strongly $k$-uniform modular sumset graph.

Conversely, assume that $G$ admits a strongly $k$-uniform modular sumset labeling, say $f$. Then, $f^+(uv)=k~\forall~uv\in E(G)$. Since, $f$ is a strong modular sumset labeling, the set-labeling number of every vertex of $G$ is a divisor of the set-labeling numbers of the edges incident on that vertex. Let $v$ be a vertex of $G$ with the set-labeling number $r$, where $r$ is a divisor of $k$, but $r^2\ne k$. Since $f$ is $k$-uniform, all the vertices in $N(v)$, must have the set-labeling number $s$, where $rs=k$. Again, all vertices, which are adjacent to the vertices of $N(v)$, must have the set-labeling number $r$. Since $G$ is a connected graph, all vertices of $G$ have the set-labeling number $r$ or $s$. Let $X$ be the set of all vertices f $G$ having the set-labeling number $r$ and $Y$ be the set of all vertices of $G$ having the set-labeling number $s$. Since $r^2\ne k$, no two elements in $X$ (and no elements in $Y$ also) can be adjacent to each other. Therefore, $G$ is bipartite. 
\end{proof}

\begin{thm}
For a positive non-square integer $k\le n$, a modular sumset labeling $f:V(G)\to \sP(\mZ)$ of an arbitrary graph $G$ is a a strongly $k$-uniform modular sumset labeling of $G$ if and only if either $G$ is bipartite or a disjoint union of bipartite components.
\end{thm}

For a positive integer $k\le n$, the maximum number of components in a strongly $k$-uniform modular sumset graph is as follows.

\begin{prop}
Let $f$ be a strongly $k$-uniform modular sumset labeling of a graph $G$ with respect to the ground set $\mZ$. Then, the maximum number of components in $G$ is the number of distinct pairs of divisors  $r$ and $s$ of $k$ such that $rs=k$.
\end{prop}

\begin{rem}{\rm 
It can be observed that a strongly $k$-uniform modular sumset graph can have a non-bipartite component if and only if $k$ is a perfect square.  More over, a strongly $k$-uniform modular sumset graph $G$ can have at most one non-bipartite component.}
\end{rem}

\subsection{Maximal Modular Sumset Graph}

It can be observed that the maximum value of the set-labeling number of an edge of a modular sumset graph is $n$, the cardinality of the ground set $\mZ$. Hence, we introduce the following notion.

\begin{defn}{\rm 
Let $f:V(G)\to \sP(\mZ)$ be a modular sumset labeling of a given graph $G$. Then, $f$ is said to be a {\em maximal modular sumset labeling} of $G$ if and only if $f^+(E(G))=\{\mZ\}$.}
\end{defn}

In other words, a modular sumset labeling $f:V(G)\to \sP(\mZ)$ of a given graph $G$ is a maximal modular sumset labeling of $G$ if the set-label of every edge of $G$ is the ground set $\mZ$ itself.

What are the conditions required for a graph to admit a maximal modular sumset labeling? Let us proceed to find out the solutions to this problem.

\begin{prop}
The modular sumset labeling $f:V(G)\to \sP(\mZ)$ of a given graph $G$ is a maximal modular sumset labeling of $G$ if and only if for every pair of adjacent vertices $u$ and $v$ of $G$ some or all of the following conditions hold.
\begin{enumerate}[itemsep=0mm]
\item[(i)] $|f(u)|+|f(v)|\ge n$ if $D_{f(u)}\cap D_{f(v)}\ne \emptyset$. The strict inequality hold when $D_{f(u)}$ and $D_{f(v)}$ are arithmetic progressions containing the same elements.
\item[(ii)] $|f(u)|\,|f(v)|\ge n$ if $D_{f(u)}\cap D_{f(v)}= \emptyset$. 
\end{enumerate}
\end{prop}
\begin{proof}
For two adjacent vertices $u$ and $v$ in $G$, let $D_{f(u)}=D_{f(v)}=\{d\}$ are arithmetic progressions containing the same elements. Then, the elements in $f(u)$ and $f(v)$ are also in arithmetic progression, with the same common difference $d$. Then, by Theorem \ref{T-SSAPL}, $|f(u)+f(v))|= |f(u)|+|f(v)|-1$. Therefore, the set-labeling number of the edge $uv$ is $n$ if and only if $|f(u)|+|f(v)|>n$. 

Now, let $D_{f(u)}\cap D_{f(v)}\ne \emptyset$ such that $D_{f(u)}\ne D_{f(v)}$. Then, clearly $|f(u)+f(v))|\ge|f(u)|+|f(v)|$. Therefore, we have $|f^+(uv)|=n$ if and only if $|f(u)|+|f(v)|\ge n$.

Next assume that $D_{f(u)}\cap D_{f(v)}= \emptyset$. Then, $|f(u)+f(v))|=|f(u)|\,|f(v)|$. Therefore, we have $|f^+(uv)|=n$ if and only if $|f(u)|\,|f(v)|\ge n$.
\end{proof}

Figure \ref{fig:G-ESSLG} illustrates a maximal modular sumset labeling of a graph with respect to the ground set $\mathbb{Z}_4$.

\begin{figure}[h!]
\centering
\includegraphics[width=0.55\linewidth]{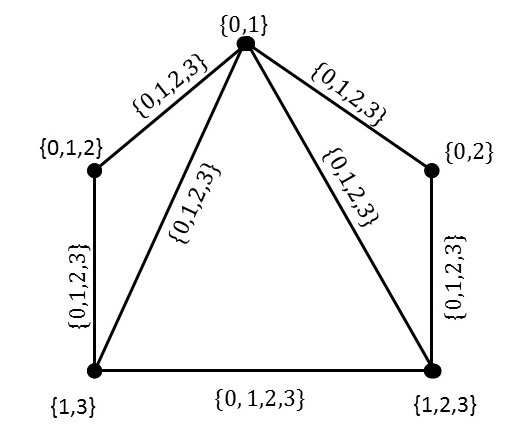}
\caption{An example to a maximal modular sumset graph.}
\label{fig:G-ESSLG}
\end{figure}

The following result explains a necessary and sufficient condition for a weak modular sumset labeling of a given graph $G$ to be a maximal modular sumset labeling of $G$.

\begin{prop}\label{T-WMSL}
A weak modular sumset labeling of a graph $G$ is a maximal labeling of $G$ if and only if $G$ is a star graph. 
\end{prop}
\begin{proof}
Let $f$ be a weak modular sumset labeling of  given graph $G$. First, assume that $f$ is a maximal modular sumset labeling of $G$. Then, the set-labeling number of one end vertex of every edge of $G$ is $1$ and the set-labeling number of the other end vertex is $n$. Therefore, $\mZ$ be the set-label of one end vertex of every edge of $G$, which is possible only if $G$ is a star graph with the central vertex has the set-label $\mZ$ and the pendant vertices of $G$ have distinct singleton set-labels.

Conversely, assume that $G$ is a star graph. Label the central vertex of $G$ by the ground set $\mZ$ and label other vertices of $G$ by distinct singleton subsets of $\mZ$. Then, all the edges of $G$ has the set-indexing number $n$. That is, this labeling is a maximal modular sumset labeling of $G$.
\end{proof}

A necessary and sufficient condition for a strong modular sumset labeling of a graph $G$ to be a maximal modular sumset labeling of $G$.

\begin{thm}
Let $f$ be a strong sumset-labeling of a given graph $G$. Then, $f$ is a maximal sumset-labeling of $G$ if and only if $n$ is a perfect square or $G$ is bipartite or a disjoint union of bipartite components. 
\end{thm}
\begin{proof}
The proof is an immediate consequence of Theorem \ref{T-SSG2}, when $k=n$.
\end{proof}

In the coming discussion, we check whether the sumset of two sets can contain both sets and according to that property we define a particular type of modular sumset graphs.

\subsection{Exquisite Modular Sumset Graphs}

Analogous to the exquisite integer additive set-labeling of a graph $G$ defined in \cite{GS3}, let us define the following.

\begin{defn}{\rm
For a positive integer $n$, let $\mZ$ be the set of all non-negative integers modulo $n$ and $\sP(\mZ)$ be its power set. An {\em exquisite modular sumset labeling} (EMSL) is a modular sumset labeling $f:V(G)\to \sP(\mZ)$ with the induced function $f^+:E(G) \to \sP(\mZ)$ defined by $f^+ (uv) = f(u)+ f(v),~ uv\in E(G)$, such that $f(u),f(v)\subseteq f^+(uv)$ for all adjacent vertices $u, v\in V(G)$. A graph which admits an exquisite modular sumset labeling is called an {\em exquisite modular sumset graph}.}
\end{defn}

What is the condition required for a given graph to admit an exquisite modular sumset labeling? The following proposition leads us to a solution to this question.

\begin{prop}\label{P-ESL1}
Let $A$ and $B$ be two subsets of the set $\mZ$. Then, $A$ (or $B$) is a subset of their sumset $A+B$ if and only if every element $a_i$ of $A$ is the sum (modulo $n$)  of an element $a_j$ (not equal to $a_i$) in $A$ and an element $b_l$ in $B$. 
\end{prop}
\begin{proof}
Let $A,B\subseteq \mZ$. Assume that every element of $A$ is the sum (modulo $n$) of an element in $A$ and an element $B$. That is, $a_i\in A \implies~ \exists ~ a_j \in A, b_l \in B ~\text{such that}~ a_i=a_j+b_l$. Hence $a_i \in A+B$. Therefore, $A\subseteq A+B$.

Conversely, assume that $A\subseteq A+B$. Then, $a_i \in A \implies a_i \in A+B \implies a_i=a_j+b_l$ for some $a_j\in A, b_l\in B$. 
\end{proof}

\noindent In view of the Proposition \ref{P-ESL1}, we have the following theorem.

\begin{thm}
A graph $G$ admits an exquisite modular sumset labeling if and only if every element in the set-label of any vertex of $G$ is the sum (modulo $n$) of an element in that set-label and an element in the set-label of its adjacent vertex.
\end{thm}

Analogous to the corresponding result of exquisite IASL-graphs, we have the following result for an exquisite modular sumset graphs.

\begin{prop}
A graph admits an exquisite modular sumset labeling $f$ if the set-label of every vertex of $G$ consists of the element $0$.
\end{prop}

\begin{prop}
Every maximal modular sumset labeling of a graph $G$ is also an exquisite modular sumset labeling of $G$.
\end{prop}
\begin{proof}
Let $f$ be a maximal sumset-labeling of $G$. Then, $f^+(uv)=\mZ$ for all edge $uv\in E(G)$. Then, the set-labels $f(u)$ and $f(v)$ are the subsets of $f^+(uv)$. Hence, $f$ is an exquisite modular sumset labeling of $G$.
\end{proof}

Invoking the above results, a necessary and sufficient condition for a modular sumset labeling of a graph $G$ to be an exquisite modular sumset labeling of $G$.

\begin{thm}\label{T-ESL1a}
A modular sumset labeling $f:V(G)\to \sP(\mZ)$ of a graph $G$ is an exquisite modular sumset labeling of $G$ if and only if either the set-label of every vertex consists of the element $0$ or $f$ is a maximal modular sumset labeling of $G$.
\end{thm}

Figure \ref{fig:G-ESSLG0} illustrates an exquisite modular sumset labeling of a graph, the set-labels of all whose vertices containing $0$. 

\begin{figure}[h!]
\centering
\includegraphics[width=0.5\linewidth]{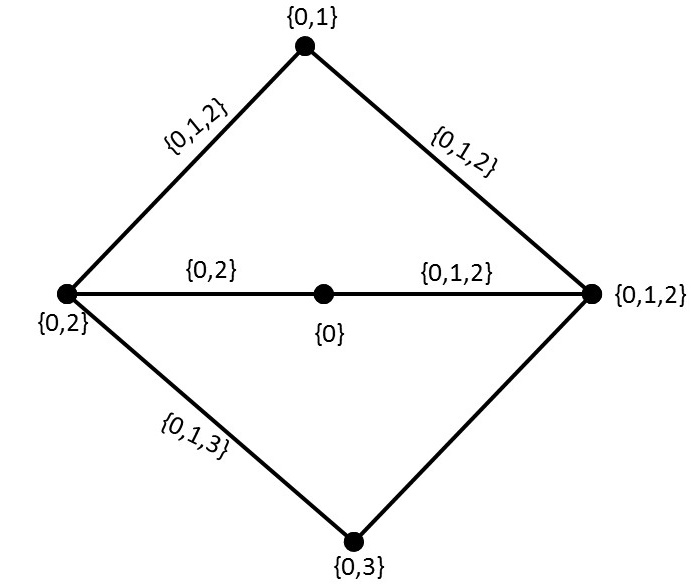}
\caption{An example to a maximal modular sumset graph.}
\label{fig:G-ESSLG0}
\end{figure}

Can a weak modular sumset labeling of a graph $G$ be an exquisite modular sumset labeling of $G$? Let us now proceed to find the solution to this problem.

\begin{thm}
Let $f:V(G)\to \sP(\mZ)$ be a weak modular sumset labeling of a graph $G$. Then, $f$ is an exquisite modular sumset labeling of $G$ if and only if $G$ is a star.
\end{thm}
\begin{proof}
Since $f$ is a weak modular sumset labeling of $G$, then at least one end vertex of every edge of $G$ must have singleton set-labels. First, assume that $G$ is a star graph. Label the central vertex by $\{0\}$ and label all pendant vertices by the subsets of $\mZ$ containing $0$. Then, this labeling is an exquisite modular sumset labeling. (Or label the central vertex of $G$ by $\mZ$ and label the pendant vertices of $G$ by distinct singleton subsets of $\mZ$. This labeling is also an exquisite modular sumset labeling of $G$.)

Conversely, assume that $f$ is an exquisite modular sumset labeling of $G$. Then by Theorem \ref{T-ESL1a}, the set-labels of all vertices of $G$ consist of $0$ or $f$ is a maximal modular sumset labeling.  If the set-labels of all vertices of $G$ contain $0$, then $\{0\}$ is the only singleton set that can be the set-label of a vertex of $G$. Therefore, $G$ must be a star in which the set-label of the central vertex is $\{0\}$ and the pendant vertices of $G$ have distinct non-singleton subsets of $\mZ$ as their set-labels. 

If $0$ is not an element of the set-label of every vertex of $G$, then $f$ is a maximal modular sumset labeling of $G$. Then, by Theorem \ref{T-WMSL}, $G$ is a star graph in which the set-label of the central vertex is $\mZ$ and the pendant vertices of $G$ have distinct singleton subsets of $\mZ$ as their set-labels.  This completes the proof.
\end{proof}

\section{Conclusion}

In this paper, we have introduced the notion of modular sumset labeling for given graphs and discussed certain properties and characteristics of modular sumset graphs. More properties and characteristics of various modular sumset graphs, both uniform and non-uniform, are yet to be investigated. Some promising problems in this area are the following.

\begin{prob}{\rm 
Find the minimum cardinality of the ground set $\mZ$ so that the modular sumset labeling $f:V(G)\to \sP(\mZ)$ of a given graph $G$ is a uniform modular sumset labeling of $G$.}
\end{prob}

\begin{prob}{\rm 
Characterise the graphs admitting the modular sumset labelings which are exquisite as well as weak.}
\end{prob}

\begin{prob}{\rm 
Find the minimum cardinality of the ground set $\mZ$ so that the modular sumset labeling $f:V(G)\to \sP(\mZ)$ defined on a given graph $G$ is a modular sumset indexer of $G$}.
\end{prob}

\begin{prob}{\rm 
Verify whether the existence of induced modular sumset labelings for certain graphs associated to the given modular sumset graphs like line graphs, total graphs etc.}
\end{prob}

\begin{prob}{\rm 
Determine the strong modular sumset number of different graph classes.}
\end{prob}

The problems of establishing the necessary and sufficient conditions for various graphs and graph classes to have certain other types of sumset valuations are also open. Studies on those sumset valuations which assign different sets having specific properties, to the elements of a given graph are also noteworthy.

\section*{Acknowledgement}

The author dedicates this paper to his research supervisor Prof. (Sr.) K. A. Germina, honouring her outstanding research career of over two decades.

\end{document}